\documentclass[12pt,letterpaper,reqno]{amsart}
\usepackage[top=1in, bottom=1in, left=1in, right=1in]{geometry}
\usepackage{enumerate}
\usepackage{amsmath,amsthm,amssymb,amsfonts,latexsym}
\usepackage{mathtools}
\usepackage{hyperref}
\usepackage[gen]{eurosym}
\usepackage[german,english]{babel}
\usepackage{dsfont}
\usepackage{tikz}
\usetikzlibrary{calc}
\usetikzlibrary{positioning}
\allowdisplaybreaks
\usepackage{comment}

\usepackage{latexsym}

\usepackage{tikz}					

\newtheorem*{theorem*}{Theorem}

\newcommand{\lambdaVkrminleftpart}{\lambda_0(V_{k+r_{\min}}}
\newcommand{\lambdaVkrmaxrightpart}{\lambda_0(V_{k-r_{\max}}}
\newcommand{\hatakleftpart}{a_{k,1}}
\newcommand{\hatakrightpart}{a_{k,2}}
\newcommand{\widetildesinels}{\widetilde \cos_{k,r_{\min}}^{\text{l}}}
\newcommand{\widetildesiners}{\widetilde \cos_{k,r_{\max}}^{\text{r}}}

\newtheorem{theorem}{Theorem}
\newtheorem{lemma}[theorem]{Lemma}

\newtheorem{proposition}[theorem]{Proposition}

\newtheorem{remark}[theorem]{Remark}
\newtheorem{Corollary}[theorem]{Corollary}

\newcommand{\be}{\begin{equation}}
\newcommand{\ee}{\end{equation}}
\newcommand{\bea}{\begin{eqnarray*}}
	\newcommand{\eea}{\end{eqnarray*}}
\newcommand{\beq}{\begin{eqnarray}}
\newcommand{\eeq}{\end{eqnarray}}

\newtheorem{cor}[theorem]{Corollary}

\usepackage{todonotes}

\title[On the asymptotic behavior of the spectral gap ...]{On the asymptotic behavior of the spectral gap for discrete Schrödinger operators}

\subjclass[2010]{}


\author[M.~Hofmann]{Matthias Hofmann}
\author[J.~Kerner]{Joachim Kerner}
\author[M.~Pechmann]{Maximilian Pechmann}

\address{Matthias Hofmann, Lehrgebiet Numerische Mathematik, Fakult\"at Mathematik und Informatik, Fern\-Universit\"at in Hagen, 58084 Hagen, Germany}
\email{matthias.hofmann@fernuni-hagen.de}

\address{Joachim Kerner, Lehrgebiet Analysis, Fakult\"at Mathematik und Informatik, Fern\-Universit\"at in Hagen, 58084 Hagen,Germany}
\email{joachim.kerner@fernuni-hagen.de}

\address{Maximilian Pechmann, Department of Mathematics, Tennessee Technological University, Cookeville, TN 38505, USA}
\email{mpechmann@tntech.edu}

\date{\today}
\thanks{
M. Hofmann was supported by the Funda\c{c}\~ao para a Ci\^encia e a Tecnologia (FCT), Portugal, within the scope of the projects Spectral Optimal Partitions: geometric and numerical analysis, reference \href{https://doi.org/10.54499/2023.13921.PEX         }{2023.13921.PEX}.
}

\begin{document}

	\begin{abstract} In this note we elaborate on the asymptotic behavior of the spectral gap of a class of discrete Schrödinger operators defined on a path graph in the limit of infinite volume. We confirm recent results and generalize them to a larger class of potentials using entirely different methods. Notably, we also resolve a conjecture previously proposed in this context. This then yields new insights into the rate at which the spectral gap tends to zero as the volume increases.
	\end{abstract}
	
\maketitle

\section{Introduction}

\noindent This paper is devoted to the investigation of the spectral gap of discrete Schrödinger operators. Here, the spectral gap refers to the difference between the lowest two eigenvalues. At least in the continuous setting, related considerations have a long tradition: for example, important results were obtained in \cite{AB} and then in \cite{Lavine}, solving the one-dimensional version of the fundamental gap conjecture~\cite{AndrewsClutterbuckFundamentalGap}. As a matter of fact, it turns out that it is in general quite difficult to say something about how the gap changes under perturbations, for instance, under the addition of some non-negative potential. Indeed, it might already be difficult to infer the correct sign of the changes~\cite{Abramovich}. This is certainly true as long as the configuration space on which the operator is defined remains fixed; for example, one could think of a Schrödinger operator defined on an interval of fixed length \cite{AK,ACH21}. However, as demonstrated recently~\cite{KernerTauferDomain,KernerTaufer}, one is able to say more about the spectral gap for a larger class of potentials if one is working with Schrödinger operators defined on domains of increasing volume. Most importantly, in \cite{KernerTaufer} the authors discussed an interesting effect which was recently studied on the path graph \cite{KY} and which forms -- in some sense -- the starting point of our investigations. More explicitly, it was shown that the presence of (certain) compactly supported potentials drastically change the asymptotic behavior of the spectral gap when compared to the asymptotic behavior of the gap of the Laplacian without a potential. In other words, as soon as one adds such a non-negative compactly supported potential to the Laplacian, the spectral gap converges strictly faster to zero in the limit of infinite volume. This somewhat surprising effect is the result of an effective degeneracy of the lowest two eigenvalues at infinite volume. It is important to note, however, that this effect might disappear in higher dimensions or reappear as well, depending on the particular nature of the potential~\cite{KernerTauferDomain}. This also explains why the authors of \cite{KY} focused on a specific discrete graph -- the path graph: adding a compactly supported potential to the Laplacian then leads, at least on an intuitive level, to the mentioned (approximate) degeneracy of the lowest two eigenvalues in the infinite-volume limit since the potential cuts the graph into two (approximately) congruent pieces.

Starting with the results of~\cite{KY}, we want to achieve two things with this paper: First, since the authors of \cite{KY} considered only a very specific potential localized on only one vertex, we generalize their main result \cite[Theorem~6]{KY} to general non-negative and compactly supported potentials, while using entirely different proof techniques. Furthermore, we prove the conjecture put forward in \cite[Section~4]{KY}, which provides more information on how fast the spectral gap converges to zero. By doing this, we also provide some rigorous justification to the numerical results obtained in \cite[Section~4]{KY}.

For completeness, let us also mention that investigations of the spectral gap (or spectral properties) for Jacobi operators \cite{Moerbeke,Teschl00} and Schrödinger-type operators on graph-like structures such as combinatorial graphs \cite{KellerLenzGraphs} and metric graphs \cite{BerkolaikoKuchmentBook,KurasovBook} are quite common. Regarding the discrete case, let us refer to \cite{FCLP18,LenzDirichlet,JostMulasMuench,BrSe24}; for the metric case, one may consult \cite{KM13,KKMM16,BandLevy,RohlederTree,BKKM17} and references therein.

Our paper is organized as follows: In Section~\ref{SecModel} we introduce the basic setting and the class of Schrödinger operators considered. In Section~\ref{SectionEstimates} we then derive upper and lower bounds on the two lowest eigenvalues. Finally, in Sections~\ref{SectionMainResultsI} and~\ref{SectionMainResultsII}, we use these estimates to prove our main results (Theorem~\ref{Theorem bounds accumulation points general potential} and Theorem~\ref{Theorem convergence}) and hereby resolve a conjecture from \cite[Section~4]{KY}, see in particular Corollary~\ref{C_1 proof conjecture special case J = 0} and Theorem~\ref{Theorem convergence}. In the Appendix~\ref{Appendix}, we derive auxiliary results that we use in Section~\ref{SectionMainResultsII}.
\section{The model}\label{SecModel}
	
\noindent Our configuration space is the path graph $G_k=(V_k,E_k)$ with vertex set $V_k:=\{-k, -k+1, \ldots, k\}$ and hence with an odd number of vertices $|V_k|= 2k+1$, $k \in \mathds{N}=\{1,2,...\}$. Except for the outer two vertices that have one neighbor each, it is readily clear that all other vertices have exactly two neighbors.  This implies that the edge set is given by $E_k=\{ \{v,w\}: v,w \in  V_k \ \text{with} \ |v-w|=1 \}$.
	
In the following, we study operators defined on the finite-dimensional Hilbert space $\mathcal{H}_k=\mathds{C}^{|V_k|}$, $k \in \mathds N$. The discrete (standard, unweighted) Laplacian $\mathcal{L}_{k}:\mathds{C}^{|V_k|} \rightarrow \mathds{C}^{|V_k|}$ is defined via 
	\begin{equation*}
	(\mathcal{L}_kf)(v):=\sum_{w \in V_k}\gamma_{w,v}(f(v)-f(w))\ , \quad v \in V_k\ , \quad f \in \mathcal{H}_k, \quad k \in \mathds N \ ,
	\end{equation*} 
where $\gamma_{v,w} = 1$ for all  $v,w \in V_k$ with $|v-w| = 1$ and $\gamma_{v,w} =0$ else. The associated quadratic form is given by
	\begin{equation*}\label{QForm}
	q_k[f]:=\frac{1}{2}\sum_{v,w \in V_k}\gamma_{w,v}|f(v)-f(w)|^2 =\frac{1}{2}\sum_{v,w \in V_k}\gamma_{w,v}|\nabla_{v,w}f|^2\ , \quad f \in \mathcal{H}_k\ ,
	\end{equation*}
	where we defined $\nabla_{v,w}f:=f(v)-f(w)$. Since the Laplacian $\mathcal{L}_k$ is a self-adjoint (non-negative) operator on a finite-dimensional Hilbert space, its spectrum consists of (non-negative, real) eigenvalues only. Furthermore, inserting the vector $\mathds 1 =(1,...,1)^T \in \mathcal H_k$ into $q_k[\cdot]$ shows that the lowest eigenvalue of $\mathcal{L}_k$, for all $k \in \mathds N$, is zero. 
	
	In a next step we construct a discrete Schrödinger operator by introducing a non-negative external potential with compact support. More explicitly, the quadratic form of our Schrödinger operator shall be given by
	\begin{equation}\label{QFormII}
	q_{k,\boldsymbol{\alpha}}[f]:=\frac{1}{2}\sum_{v,w \in V_k}\gamma_{w,v}|\nabla_{v,w}f|^2+\sum\limits_{j \in J} \alpha_j |f(j)|^2\ , \quad f \in \mathcal{H}_k \ ,
	\end{equation}
	where $\boldsymbol{\alpha} = (\alpha_j)_{j \in J}$ and $J \subset \mathds Z$ is a non-empty set independent of $k$; $\alpha_j > 0$, $j \in J$, represents the strength of the external potential localized at the $j$-th vertex. The associated self-adjoint operator then becomes
    \begin{equation}
        H_{k,\boldsymbol{\alpha}}:=\mathcal{L}_k + \sum\limits_{j \in J} \alpha_j \delta_j \ , \quad k \in \mathds N \ ,
    \end{equation} 
    where $\delta_j$ is such that $(\delta_j f)_{i}:=\delta_{ij}f(j)$, $i \in V_k$, $\delta_{ij}$ referring to the Kronecker delta. As in the case for $\mathcal L_k$, the operator $H_{k,\boldsymbol{\alpha}}$ is, for all $k \in \mathds N$, non-negative and its spectrum consists of non-negative eigenvalues only. For all $k \in \mathds N$, we denote the eigenvalues of $H_{k,\boldsymbol{\alpha}}$ by
    \begin{equation*}
    \lambda_0(V_k, \boldsymbol \alpha) < \lambda_1(V_k, \boldsymbol \alpha) \leq ... \leq \lambda_{|V_k|-1}(V_k, \boldsymbol \alpha) \ . 
    \end{equation*}
	We can now introduce the central object of interest in this paper, which is the spectral gap
	\begin{equation} \label{definition sequence of interest}
	\Gamma(V_k, \boldsymbol \alpha):=\lambda_1(V_k, \boldsymbol \alpha)-\lambda_0(V_k, \boldsymbol \alpha)\ .
	\end{equation}
    Note that $\Gamma(V_k, \boldsymbol \alpha) > 0$ for each value of $k \in \mathds{N}$, since the ground state is non-degenerate; more on the spectral theory of graphs can be found in~\cite{Chung,Brouwer,KellerLenzGraphs}. We denote by $\varphi_{k, \boldsymbol{\alpha}}$ the normalized ground state of $H_{k,\boldsymbol{\alpha}}$, that is, the eigenvector corresponding to the lowest eigenvalue $\lambda_0(V_k, \boldsymbol \alpha)$ with $\|\varphi_{k, \boldsymbol{\alpha}}\|^2_{\mathcal{H}_k}=\sum_{j=-k}^k |\varphi_{k, \boldsymbol{\alpha}}(j)|^2 = 1$. Note that the ground state $\varphi_{k, \boldsymbol{\alpha}}$ can be chosen to be strictly positive.
    
    We set
        \begin{equation*}
            \widetilde \alpha \coloneq \min\limits_{j \in J} \alpha_j \qquad \text{ and } \qquad \widehat \alpha \coloneq \sum\limits_{j \in J} \alpha_j
        \end{equation*}
and frequently use the identity, $k \in \mathds N$,
\begin{equation*}
    \frac{1}{2}\sum_{v,w \in V_k}\gamma_{w,v}|\nabla_{v,w}f|^2=\sum_{j=-k}^{k-1}|\nabla_{j}f|^2=\sum_{j=-k}^{k}|\nabla_{j}f|^2
\end{equation*}
where $\nabla_{j}f:=f(j+1)-f(j)$ and where we set $f(k+1):=f(k)$.

For the special case $J = \{0\}$ with external potential $\alpha \delta_0$, $\alpha >0$, we write
\begin{align*}
H_{k,\alpha} \coloneq \mathcal L_k + \alpha \delta_0 \ , \quad k \in \mathds N \ ,
\end{align*}
for the Hamiltonian, $\varphi_{k, \alpha}$ for its ground state, and $\Gamma(V_k, \alpha)$ for the spectral gap.
It is important to mention that in this case, one has 
\begin{equation} \label{lambda1alphadelta0 equation}
    \lambda_1(V_k,\alpha \delta_0)  = \lambda_0(V_k,\infty \delta_0)
\end{equation}
for all $\alpha \ge 0$ and all $k \in \mathds N$. Here, since the external potential is $\alpha\delta_0$, we write $\lambda_1(V_k,\alpha \delta_0)$ for $\lambda_1(V_k, \boldsymbol \alpha)$; similarly, $\lambda_0(V_k,\infty \delta_0)$ is shorthand for the lowest eigenvalue of the Laplacian on the path graph with a Dirichlet condition at the zero vertex. In the same spirit, we use $\lambda_0(V_k, \alpha \delta_0)$ as shorthand for $\lambda_0(V_k, \boldsymbol \alpha)$ in case of an external potential $\alpha\delta_0$. Later we shall also use that, see for example \cite{Brouwer},
\begin{equation} \label{Gleichung lambda0}
    \lambda_0(V_k,\infty \delta_0) = 2- 2\cos\left(\pi/|V_k|\right) \ , \quad k \in \mathds N \ ,
\end{equation}
thus, in particular, 
%
\begin{align} \label{obere Schranke lambda0inftydelta0}
    \lambda_0(V_k,\infty \delta_0) =\pi^2|V_k|^{-2} + \mathcal O(|V_k|^{-4}) \ .
\end{align}
%

%

We conclude this section with two remarks. The first is intended to clarify our notation used in this paper, while the second presents a physical intuition that motivated and guided our analysis.

    \begin{remark}
       For two sequences $(a_k)_{k \in \mathds N}$, $(b_k)_{k \in \mathds N}$, $a_k \neq 0$ for all but finitely many $k \in \mathds N$, we use the notation $b_k \lesssim a_k$ iff there is a constant $c> 0$ such that $|b_k/a_k| \le c$ for all but finitely many $k \in \mathds N$, and $b_k \sim a_k$ iff there are constants $c,C >0$ such that $c \le |b_k/a_k| \le C$ for all but finitely many $k \in \mathds N$.

        Also, in the proofs and some statements we denote by $(const.)$ positive constants that are independent of $k$ and the potential strength $\boldsymbol{\alpha} = (\alpha_j)_{j \in J}$, whenever their precise values are irrelevant. Note, however, that they could still depend on the support of the potential.
    \end{remark}
\begin{remark}\label{RemarkPechmannPhysik} We encourage the reader to keep the following physical intuition in mind, which also guided the development of this work. Our analysis focuses on the two lowest eigenvalues, both of which converge to zero as the size of the graph tends to infinity. In this low-energy regime, a particle occupying a quantum state corresponding to either eigenvalue has arbitrarily small kinetic energy. Consequently, any potential of fixed finite strength, informally, becomes effectively impenetrable to such a particle, and therefore behaves, in the limit, as if it were infinitely strong. This perspective helps to explain the appearance of quantities related to infinitely strong potentials in our results.
\end{remark}
\section{Spectral estimates}\label{SectionEstimates}
\noindent In this section we provide upper and lower bounds for $\lambda_0(V_k, \boldsymbol \alpha)$ as well as for $\lambda_1(V_k, \boldsymbol \alpha)$. We start with a lower bound for $\lambda_0(V_k, \boldsymbol \alpha)$. We set
\begin{align*}
r_{\min} \coloneq \min \{j \in J\} \qquad &\text{ and } \qquad r_{\max} \coloneq \max \{j \in J\} \ ,
\end{align*}
and introduce
\begin{align}
    \hatakleftpart &\coloneq \dfrac{1}{2} - \sum_{j=-k}^{r_{\min}} |\varphi_{k, \boldsymbol{\alpha}}(j) - \varphi_{k, \boldsymbol{\alpha}}(r_{\min})|^2 \label{definition ak1}\ ,\\
\shortintertext{as well as}
\hatakrightpart &\coloneq \dfrac{1}{2} - \sum_{j=r_{\max}}^{k} |\varphi_{k, \boldsymbol{\alpha}}(j) - \varphi_{k, \boldsymbol{\alpha}}(r_{\max})|^2 \label{definition ak2} \ .
\end{align}
Note that $1 \ge \hatakleftpart + \hatakrightpart \ge 1 - \sum_{j=-k}^k |\varphi_{k, \boldsymbol{\alpha}}(j)|^2 = 0$ due to monotonicity of the ground state outside of $[r_{\min}, r_{\max}]$.
\begin{theorem}[Lower bound ground state energy] \label{Theorem lower bound}
Consider a Hamiltonian $H_{k,\boldsymbol{\alpha}}$ with ground state energy $\lambda_0(V_k, \boldsymbol \alpha) $. Then, for all $k > \max\{|r_{\min}|, |r_{\max}|\}/2$, we have
\begin{equation}\label{LowerBoundEigenvalue}
\begin{split}
\lambda_0(V_{k}, \boldsymbol \alpha) 
& \ge \left( \dfrac{1}{2} - \hatakleftpart \right) \lambdaVkrminleftpart, \infty \delta_0)  + \left( \dfrac{1}{2} - \hatakrightpart \right) \lambdaVkrmaxrightpart, \infty \delta_0) \\
&\qquad + \sum\limits_{j \in J} \alpha_j | \varphi_{k, \boldsymbol{\alpha}}(j)|^2 
\ .
\end{split}
\end{equation}
Moreover,
\begin{equation}\label{UpperBoundAKs}
0 \leq \hatakleftpart + \hatakrightpart \le (const.) \widetilde \alpha^{-1} k^{-1}
\end{equation}
for all but finitely many $k \in \mathds N$.
\end{theorem}
\begin{proof} In a first step we obtain, using the minmax-principle,
    \begin{align*}
            \lambda_0(V_{k}, \boldsymbol \alpha) 
            &\ge \sum\limits_{j=-k}^{r_{\min}-1} |\nabla_j (\varphi_{k, \boldsymbol{\alpha}} - \varphi_{k, \boldsymbol{\alpha}}(r_{\min}))|^2 + \sum\limits_{j=r_{\max}}^{k} |\nabla_j ( \varphi_{k, \boldsymbol{\alpha}} - \varphi_{k, \boldsymbol{\alpha}}(r_{\max}))|^2\\
            &\qquad + \sum\limits_{j \in J} \alpha_j | \varphi_{k, \boldsymbol{\alpha}}(j)|^2 \\
            &\ge \sum\limits_{j=-k}^{r_{\min}} |\varphi_{k, \boldsymbol{\alpha}}(j) - \varphi_{k, \boldsymbol{\alpha}}(r_{\min})|^2  \lambdaVkrminleftpart, \infty \delta_0) \\
            & \qquad + \sum\limits_{j=r_{\max}}^k |\varphi_{k, \boldsymbol{\alpha}}(j) - \varphi_{k, \boldsymbol{\alpha}}(r_{\max})|^2 \lambdaVkrmaxrightpart, \infty \delta_0) + \sum\limits_{j \in J} \alpha_j | \varphi_{k, \boldsymbol{\alpha}}(j)|^2 \ .
        \end{align*}
        Using the definition of $a_{k,1}, a_{k,2}$ then implies the lower bound~\eqref{LowerBoundEigenvalue}.
        
        We now proceed to prove \eqref{UpperBoundAKs}: Since $2\lambda_0(V_{k}, \boldsymbol \alpha) \le \lambdaVkrminleftpart,\infty \delta_0) + \lambdaVkrmaxrightpart,\infty \delta_0)$,
we conclude
\begin{align} \label{inequality asdsad23123lxc}
    \sum\limits_{j \in J} \alpha_j | \varphi_{k, \boldsymbol{\alpha}}(j)|^2 \le \hatakleftpart\lambdaVkrminleftpart, \infty \delta_0) + \hatakrightpart\lambdaVkrmaxrightpart, \infty \delta_0)\ ,
\end{align}
and
\begin{align} \label{inequality widetilde alpha lkj2kl3j}
& \sum\limits_{j \in J} |\varphi_{k, \boldsymbol{\alpha}}(j)|^2 \le \widetilde \alpha^{-1} \Big( \hatakleftpart\lambdaVkrminleftpart, \infty \delta_0) + \hatakrightpart\lambdaVkrmaxrightpart, \infty \delta_0) \Big) \ .
\end{align}
Next, we observe that
\begin{align*}
1 = \sum\limits_{j=-k}^{k} |\varphi_{k, \boldsymbol{\alpha}}(j)|^2 = \sum\limits_{j=-k}^{r_{\min}-1} |\varphi_{k, \boldsymbol{\alpha}}(j)|^2 + \sum\limits_{j=r_{\max}+1}^{k} |\varphi_{k, \boldsymbol{\alpha}}(j)|^2 + \sum\limits_{j=r_{\min}}^{r_{\max}} |\varphi_{k, \boldsymbol{\alpha}}(j)|^2 
\end{align*}
and therefore
\begin{align} \label{equality for ak}
\begin{split}
\hatakleftpart + \hatakrightpart 
& = \sum\limits_{j=r_{\min}}^{r_{\max}} |\varphi_{k, \boldsymbol{\alpha}}(j)|^2  + 2 \varphi_{k, \boldsymbol{\alpha}}(r_{\min}) \sum\limits_{j=-k}^{r_{\min}-1} \varphi_{k, \boldsymbol{\alpha}}(j) + 2 \varphi_{k, \boldsymbol{\alpha}}(r_{\max}) \sum\limits_{j=r_{\max}+1}^k \varphi_{k, \boldsymbol{\alpha}}(j) \\
& \qquad \qquad  - |\varphi_{k, \boldsymbol{\alpha}}(r_{\min})|^2 (k+r_{\min}) - |\varphi_{k, \boldsymbol{\alpha}}(r_{\max})|^2 (k-r_{\max}) \ .
\end{split}
\end{align}
To simplify we set $\Theta_{\max}:=\max\big\{ \lambdaVkrminleftpart, \infty \delta_0), \lambdaVkrmaxrightpart, \infty \delta_0) \big\}$. Then, using \eqref{inequality widetilde alpha lkj2kl3j} and Hölder's inequality as well as $a+b \leq \max\{2a,2b\}$ for $a,b > 0$, we get 
\begin{align*}
\hatakleftpart + \hatakrightpart 
& \le (const.) \max \Big\{ \widetilde \alpha^{-1/2} \left(\hatakleftpart + \hatakrightpart \right)^{1/2} \Theta_{\max}^{1/2} k^{1/2}, \sum\limits_{j=r_{\min}}^{r_{\max}} |\varphi_{k, \boldsymbol{\alpha}}(j)|^2\Big\} \ .
\end{align*}
In order to estimate the last sum, we use that for all $j \in\{r_{\min}, \ldots, r_{\max}\}$, since $(a+b)^2 \leq 2a^2 + 2b^2$ for $a,b \in \mathds R$ and due to the eigenvalue equation on the form level, 
\begin{equation}
\begin{aligned}
   |\varphi_{k, \boldsymbol{\alpha}}(j)|^2 &\le \left (|\varphi_{k, \boldsymbol{\alpha}}(r_{\min})| + \sum_{i=r_{\min}}^{j-1} |\varphi_{k, \boldsymbol{\alpha}}(i+1) - \varphi_{k, \boldsymbol{\alpha}}(i)| \right )^2\\
    &\le  2\max_{j\in J} |\varphi_{k, \boldsymbol{\alpha}}(j)|^2 + 2 |r_{\max}-r_{\min}| \lambda_0(V_k, \boldsymbol \alpha)\ .
\end{aligned}
\end{equation}
Therefore, using again the inequalities $a+b \leq \max\{2a,2b\}$ for $a,b > 0$ and \eqref{inequality widetilde alpha lkj2kl3j}, we get
\begin{equation} \label{upper bound hatakleftpart plus hatakrightpart general case}\begin{split}
\hatakleftpart + \hatakrightpart & \le (const.) \max\Big\{ \widetilde \alpha^{-1} \Theta_{\max} k, \Theta_{\max} \Big\}
\end{split}
\end{equation}
for all but finitely many $k \in \mathds N$. The claimed estimate~\eqref{UpperBoundAKs} now follows by observing that $\Theta_{\max} \sim k^{-2}$. 




%
%
\end{proof}

\begin{remark} \label{Remark upper bound phik(0)}
Inspecting the proof of Theorem~\ref{Theorem lower bound}, in particular \eqref{inequality asdsad23123lxc}, we obtain the following estimate for the special case when $J = \{0\}$, that is, for $H_{k,\alpha} = \mathcal L_k + \alpha \delta_0$ with $\alpha > 0$: For all but finitely many $k \in \mathds N$,
\begin{equation} \label{upper bound phikalpha}
    \varphi_{k, \alpha}(0) \leq (const.) \alpha^{-1} k^{-3/2} \ .
\end{equation}
This can be regarded as a generalization of~\cite[Lemma~5]{KY}. Note here that $\varphi_{k, 0}(0)= |V_k|^{-1/2}$ (that is for $\alpha=0$).

From a physics point of view, it is interesting to remark that the just mentioned estimate implies an upper bound on the potential energy of a particle described by $H_{k,\alpha} = \mathcal L_k + \alpha \delta_0$ with $\alpha > 0$: More explicitly, the potential energy in the ground state is given by $E_{pot}(\alpha):=\alpha |\varphi_{k, \alpha}(0)|^2$ and we therefore obtain, for $k \in \mathds{N}$,
$$E_{pot}(\alpha) \leq (const.) \alpha^{-1} k^{-3} \ .$$
The interesting point here is that the potential energy is of a lower order than the kinetic energy, which is of order $k^{-2}$. 
\end{remark}

In a next statement we prove an upper bound for the lowest eigenvalue $\lambda_0(V_k, \boldsymbol \alpha)$. In order to formulate it, we introduce the two functions
\begin{equation}\label{Definition cossinus tilde functionen}
\begin{split}
\widetildesinels(j) & \coloneq \begin{cases}A_{\min}^{-1/2}\cos \left(\left(j+k+\frac{1}{2}\right) \pi/(2(k+r_{\min})+1)\right), \ -k\leq j \leq r_{\min}\ , \\
0\ \quad \text{else}\ ,
\end{cases} \\
\widetildesiners(j) &\coloneq \begin{cases} A_{\max}^{-1/2}\cos \left(\left(k-j+\frac{1}{2}\right) \pi/(2(k-r_{\max})+1)\right), \ r_{\max}\leq j \leq k\ ,\\
0\ \quad \text{else}\ .
\end{cases}
\end{split}
\end{equation}
Here, $A_{\max},A_{\min} > 0$ are such that both functions are normalized to $1/2$. Also note that both functions are constructed using the well-known explicit form of the second eigenfunction to the Hamiltonian $H_{k,\boldsymbol{\alpha}} = \mathcal L_k + \alpha\delta_0$. In particular, $\varphi_{k,\infty} = \widetilde \cos_{k,0}^{\text{l}} + \widetilde \cos_{k,0}^{\text{r}}$ where $\varphi_{k,\infty}$ denotes the ground state of $H_{k,\alpha}$ for $\alpha=\infty$.
\begin{theorem}[Upper bound ground state energy I] \label{Theorem upper boundII}
Consider a Hamiltonian $H_{k,\boldsymbol{\alpha}}$ with ground state $\varphi_{k, \boldsymbol{\alpha}} \in \mathcal{H}_k$ corresponding to the ground-state energy $\lambda_0(V_k, \boldsymbol \alpha) $. Then
\begin{align} \label{inequality upper bound ground state energy}
    \lambda_0(V_k, \boldsymbol \alpha) \le \dfrac{1}{2}(1-b_k)  \Big( \lambdaVkrminleftpart,\infty \delta_0) + \lambdaVkrmaxrightpart,\infty \delta_0) \Big) + b_k c_k^2 \ ,
\end{align}
where $c_k^2 := \lambda_0(V_k,\infty\delta_0)/ (2+\epsilon)$ for an arbitrary $\epsilon > 0$ and with the sequence $(b_k)_{k \in \mathds N}$, $0 < b_k \le 1$ for all $k \in \mathds N$, such that 
\begin{equation*}
 \psi_k(j):= \left( 1- b_k \right)^{1/2} \left( \widetildesinels(j) + \widetildesiners(j) \right) + \widehat \alpha^{-1/2} b^{1/2}_kc_k \ , \quad j \in V_k \ , \quad k \in \mathds N \ ,
\end{equation*}
is normalized. Furthermore,
\begin{align} \label{inequality lower bound bk}
b_k \ge (const.) \widehat \alpha^{-1} k^{-1}
\end{align}
for all but finitely many $k \in \mathds N$.
\end{theorem}
\begin{proof}
We set $\widehat c_k \coloneq \widehat \alpha^{-1/2} c_k$ and write
\begin{equation*}
 \psi_k(j)= \left( 1- b_k \right)^{1/2} \left( \widetildesinels(j) + \widetildesiners(j) \right) + b^{1/2}_k\widehat c_k \ , \quad j \in V_k \ , \quad k \in \mathds N \ .
\end{equation*}
Calculating the norm of $\psi_k$ gives the relation
 \begin{align} \label{Definitionbk}
b_k & = 2 \left( 1 - b_k \right)^{1/2} b_k^{1/2} \widehat c_k \sum\limits_{j=-k}^k \left( \widetildesinels(j) + \widetildesiners(j) \right) + |V_k| b_k \widehat c_k^2 \ .
\end{align}
This equation can be solved explicitly for $b_k$, yielding the desired sequence. Recall that $(b_k)_{k \in \mathds N}$ is such that $0 < b_k \le 1$ for all $k \in \mathds N$ and $\|\psi_k\|_{\mathcal{H}_k} = 1$ for all $k \in \mathds N$. 

Employing the minmax-principle using $\psi_k$ as a trial state then implies
\begin{align*}
    \lambda_0(V_k, \boldsymbol \alpha) & \le (1-b_k) \dfrac{1}{2} \Big( \lambdaVkrminleftpart,\infty \delta_0) + \lambdaVkrmaxrightpart,\infty \delta_0) \Big) + b_k c_k^2 \ .
\end{align*}

Lastly, we prove \eqref{inequality lower bound bk}. Note that there exists a constant $c > 0$ such that
\begin{equation} \label{inequality sum widetilde sinelr}
\sum\limits_{j=-k}^k\big(\widetildesinels(j) + \widetildesiners(j)\big) \geq c k^{1/2}
\end{equation}
for all but finitely many $k \in \mathds N$. Let $k \in \mathds N$ be sufficiently large such that \eqref{inequality sum widetilde sinelr} and $c \widehat \alpha^{-1} k^{-1} \le 3/4$ holds but otherwise arbitrary. Suppose that $b_k \ge c \widehat \alpha^{-1} k^{-1}$; then there is nothing left to show in regard to \eqref{inequality lower bound bk}. Thus, suppose $b_k < c \widehat \alpha^{-1} k^{-1} \le 3/4$. With~\eqref{Definitionbk} and~\eqref{inequality sum widetilde sinelr} we conclude $b_k \ge c b_k^{1/2} \widehat c_k k^{1/2}$. Using~\eqref{obere Schranke lambda0inftydelta0} this then implies~\eqref{inequality lower bound bk}.
\end{proof}

We conclude this section by providing upper and lower bounds for the second-lowest eigenvalue $\lambda_1(V_k, \boldsymbol \alpha)$.
\begin{lemma}[Bounds on first excited state energy] \label{Lemma 4 asdasdkk4}
        For all but finitely many $k \in \mathds N$, one has
        \begin{equation*}
    \lambda_0(V_k, \infty \delta_0) \le \lambda_1(V_k, \boldsymbol \alpha) \le \max\big\{ \lambda_0(V_{k+r_{\min}}, \infty \delta_0), \lambda_0(V_{k-r_{\max}}, \infty \delta_0) \big\}  \ .
        \end{equation*}
    \end{lemma}
\begin{proof}
The lower bound is a direct consequence of $\lambda_1(V_k, \boldsymbol \alpha ) \ge \lambda_1(V_k, \boldsymbol 0) = \lambda_0(V_k, \infty \delta_0)$. The upper bound follows from the operator inequality $\mathcal L_k + \sum_{j \in J} \alpha_j \delta_j \le \mathcal L_k + \sum_{j \in J} \infty \delta_j$ in combination with $\lambda_1(V_k, \sum_{j \in J} \infty \delta_j) \le \max\{ \lambda_0(V_{k+r_{\min}}, \infty \delta_0), \lambda_0(V_{k-r_{\max}}, \infty \delta_0) \}$, which holds for all but finitely many $k \in \mathds N$.
    \end{proof}

\section{Main results I: General Case}\label{SectionMainResultsI}
\noindent In this section, we study properties of the sequence
\begin{equation} \label{sequence of interest}
    \left(|V_k|^3 \, \Gamma(V_k, \boldsymbol \alpha)\right)_{k \in \mathds{N}}
\end{equation}
and prove a conjecture regarding its limiting behavior as put forward in \cite{KY} (additional results will be provided in Section~\ref{SectionMainResultsII}). By doing this and by considering potentials of compact support, we also generalize the main result of \cite{KY} -- both -- in terms of convergence speed and applicability to a broader class of potentials: more explicitly, we will show that 
\begin{equation*}
   \lim_{k \rightarrow \infty}\left(|V_k|^{2 + \eta} \cdot \Gamma(V_k, \boldsymbol \alpha)\right)=0
\end{equation*}
for all $0 \leq \eta < 1$, as soon as the external potential is not the zero potential. Here, one should recall that $ \lim_{k \rightarrow \infty}\left(|V_k|^{2} \cdot \Gamma(V_k, \boldsymbol \alpha)\right)=\pi^2$ for $\boldsymbol \alpha= \boldsymbol 0$; compare with \cite{KY}. In other words, as soon as there is a non-vanishing external potential of compact support, the spectral gap converges strictly faster to zero than in the absence of any potential. This might be surprising at first, since the potential is supported on a smaller and smaller fraction of the configuration space in the limit of infinite volume. Consequently, this is a rather remarkable spectral effect. 

We now state the first main result of this paper. 
\begin{theorem}[Asymptotic upper and lower bounds I] \label{Theorem bounds accumulation points general potential}
Consider a Hamiltonian $H_{k,\boldsymbol{\alpha}} = \mathcal L_k + \sum_{j \in J} \alpha_j \delta_j$ with compactly supported external potential and associated spectral gap $ \Gamma(V_k, \boldsymbol \alpha)$. Then, the following holds:
\begin{enumerate}[(i)]
\item\label{Theorem bounds accumulation points general potential part 1} There exists a constant $C=C(\boldsymbol \alpha)> 0$ such that
\begin{equation*}
   |V_k|^3 \, \Gamma(V_k, \boldsymbol \alpha) \leq C\quad \text{for all}\quad k \in \mathds{N}\ . 
\end{equation*}
\item\label{Theorem bounds accumulation points general potential part 2} If $\widehat \alpha=\sum_{j \in J} \alpha_j $ is sufficiently small or if $J = \{0\}$, there exists a constant $c=c(\boldsymbol \alpha)>0$ such that
\begin{equation*}
    |V_k|^3 \, \Gamma(V_k, \boldsymbol \alpha) \geq c\quad \text{for all}\quad k \in \mathds{N}\ . 
\end{equation*}
\end{enumerate}
\end{theorem}
\begin{proof}
 To simplify notation, we set $\Theta_{\max}:=\max\big\{ \lambda_0(V_{k+r_{\min}}, \infty \delta_0),  \lambda_0(V_{k-r_{\max}}, \infty \delta_0) \big\}$ and $\Theta_{\min}:=\min\big\{ \lambda_0(V_{k+r_{\min}}, \infty \delta_0),  \lambda_0(V_{k-r_{\max}}, \infty \delta_0) \big\}$.
 
 We start with \eqref{Theorem bounds accumulation points general potential part 1}: Theorem~\ref{Theorem lower bound} and Lemma~\ref{Lemma 4 asdasdkk4} immediately imply for all but finitely many $k \in \mathds N$,
\begin{equation} \label{equation upper bound gap}
    |V_k|^3 \, \Gamma(V_k, \boldsymbol \alpha) \le (const.) \widetilde \alpha^{-1} k^{-1} \Theta_{\max} |V_k|^3 + \left(\Theta_{\max}- \Theta_{\min}\right) |V_k|^3 \ .
\end{equation}
Now, taking into account that $\Theta_{\max} \sim k^{-2}$ and $\Theta_{\max}- \Theta_{\min} \lesssim k^{-3}$ completes the proof; recall here \eqref{Gleichung lambda0}.

We now turn to \eqref{Theorem bounds accumulation points general potential part 2}: Theorem~\ref{Theorem upper boundII} and Lemma~\ref{Lemma 4 asdasdkk4} imply for all but finitely many $k \in \mathds N$,
\begin{equation} \label{equation lower bound gap}
     |V_k|^3 \, \Gamma(V_k, \boldsymbol \alpha)  
    \ge (const.) \widehat \alpha^{-1} \lambda_0(V_k, \infty \delta_0) k^{-1} |V_k|^3 + \left( \lambda_0(V_k, \infty \delta_0) - \Theta_{\max}\right) |V_k|^3 \ .
\end{equation}
By direct calculation one has $\Theta_{\max}-\lambda_0(V_k, \infty \delta_0) \lesssim k^{-3}$. Note that the last term is negative whenever $J \neq \{0\}$ and it vanishes if and only if $J = \{0\}$. Therefore, taking \eqref{obere Schranke lambda0inftydelta0} into account, we see that $\widehat \alpha$ has to be sufficiently small for $J \neq \{0\}$ in order to obtain a non-trivial lower bound. \eqref{obere Schranke lambda0inftydelta0} also implies the statement. 
\end{proof}

\begin{remark}
We now offer a more detailed analysis and interpretation of the upper and lower bounds obtained in Theorem~\ref{Theorem bounds accumulation points general potential} and its proof.
\begin{enumerate}[(i)]
\item Regarding the upper bound~\eqref{equation upper bound gap}: Note that the first term is of the form $(const.) \widetilde \alpha^{-1}$ and consequently converges to zero when the potential strength converges to infinity in the sense that $\widetilde \alpha \to \infty$. However, the last term is independent of the potential strength and is non-zero if and only if $r_{\max} \neq - r_{\min}$. Therefore, the entire upper bound is of the form $(const.) \widetilde \alpha^{-1}$ and consequently converges to zero as $\widetilde \alpha$ converges to infinity if and only if $r_{\max} = - r_{\min}$.

\item Regarding the lower bound~\eqref{equation lower bound gap}: The last term is negative whenever $J \neq \{0\}$ and it vanishes if and only if $J = \{0\}$; note that, trivially,
$\Gamma(V_k, \boldsymbol \alpha) \ge 0$ for all $k \in \mathds N$. Therefore, this lower bound is non-trivial if and only if $\widehat \alpha$ is sufficiently small or if $J= \{0\}$. Furthermore, the first term of this bound is of the form $(const.) \widehat \alpha^{-1}$. Consequently, if and only if $J = \{0\}$, the entire lower bound is of the form $(const.)\widehat \alpha^{-1}$ and converges to zero as the strength of the potential tends to infinity in the sense that $\widehat \alpha \to \infty$. Note that the lower bound converges to $\infty$ as $\hat \alpha\to 0$. This refers to the case where we approach the free Laplacian, for which the spectral gap vanishes in lower order.

\item The reason we end up with a trivial lower bound when $J \neq \{0\}$ and $\widehat \alpha$ is too large is the lack of a sufficiently good lower bound for $\lambda_1(V_k, \boldsymbol \alpha)$ in terms of $\lambda_0(V_{k + r_{\min}}, \infty \delta_0)$ and/or $\lambda_0(V_{k - r_{\max}}, \infty \delta_0)$; compare with Lemma~\ref{Lemma 4 asdasdkk4}. Indeed, also from the physical point of view outlined in Remark~\ref{RemarkPechmannPhysik}, it is reasonable to expect that the lower bound for $\lambda_1(V_k, \boldsymbol \alpha)$ in Lemma~\ref{Lemma 4 asdasdkk4} is suboptimal for strong external potentials. Nevertheless, a corresponding statement as in $(ii)$ of Theorem~\ref{Theorem bounds accumulation points general potential} is expected for all potential strengths. 
\end{enumerate}
\end{remark}

Next, for the convenience of the reader, we summarize our findings from Theorem~\ref{Theorem bounds accumulation points general potential} in the special case $J = \{0\}$, that is, when the potential is localized on the zero vertex in the middle of the path graph. In addition, we provide upper and lower bounds for the limit that are explicit in the coupling strength. This special case was discussed in \cite{KY} and originally motivated this paper, and Corollary~\ref{C_1 proof conjecture special case J = 0} proves the conjecture that was put forward in \cite[Section~4]{KY}.

\begin{Corollary}[Asymptotic upper and lower bounds II] \label{C_1 proof conjecture special case J = 0}
Consider the Hamiltonian $H_{k,\alpha} = \mathcal L_k + \alpha \delta_0$ with arbitrary $\alpha > 0$ and associated spectral gap $\Gamma(V_k, \alpha)$. Then there exist constants $c,C > 0$ independent of $k$ and $\alpha$ such that for all but finitely many $k \in \mathds N$,
\begin{equation*}
    \dfrac{c}{\alpha} \le |V_k|^3 \, \Gamma(V_k, \alpha) \le \dfrac{C}{\alpha} \ .
\end{equation*}
\end{Corollary}
\begin{proof}
 
At first, with Theorem~\ref{Theorem lower bound} and \eqref{lambda1alphadelta0 equation} we conclude for all but finitely many $k \in \mathds N$,
\begin{equation*}
   |V_k|^3 \, \Gamma(V_k, \alpha) \le (const.) \alpha^{-1} k^{-1} \lambda_0(V_k, \infty \delta_0)|V_k|^3 \ .
\end{equation*}
Secondly, Theorem~\ref{Theorem upper boundII} and again \eqref{lambda1alphadelta0 equation} imply
\begin{equation*}
    |V_k|^3 \, \Gamma(V_k, \alpha) \ge (const.) \alpha^{-1} k^{-1} \lambda_0(V_k, \infty \delta_0)|V_k|^3
\end{equation*}
for all but finitely many $k \in \mathds N$. In a final step, recall \eqref{obere Schranke lambda0inftydelta0}.

\end{proof}

Let us remark that, in the continuous one-dimensional setting working with operators on an interval, a result similar to Corollary~\ref{C_1 proof conjecture special case J = 0} was established in \cite{KernerTaufer} for the special case of a symmetric step potential. On the other hand, in \cite{KernerCompact}, the authors could prove corresponding lower bounds for a larger class of symmetric and compactly supported potentials but those bounds are quartic in the length and do not, as conjectured in \cite{KernerTaufer}, involve the length to the power three.

\section{Main results II: Convergence in the Special Case $J = \{0\}$}\label{SectionMainResultsII}

\noindent We conclude this paper with a convergence result for the case $J = \{0\}$, that is, when
 \begin{equation*}
        H_{k,\alpha}:=\mathcal{L}_k + \alpha \delta_0 \ , \quad k \in \mathds N \ ,
    \end{equation*} 
with arbitrary $\alpha > 0$. Recall that we denote the ground state of $H_{k,\alpha}$ by $\varphi_{k, \alpha} \in \mathcal{H}_k$, and its two lowest eigenvalues by $\lambda_0(V_k, \alpha \delta_0)$ and $\lambda_1(V_k, \alpha \delta_0)$, respectively. Consequently, the corresponding spectral gap reads $\Gamma(V_k,\alpha) = \lambda_1(V_k, \alpha \delta_0) - \lambda_0(V_k, \alpha \delta_0)$. Using auxiliary results established in the appendix, we in particular show convergence of the sequence~\eqref{sequence of interest} in this special case. Note that this result provides us with a strengthening of Theorem~\ref{Theorem bounds accumulation points general potential} and Corollary~\ref{C_1 proof conjecture special case J = 0} in this setting. We also remark that this convergence was suggested in \cite[Section~4]{KY} based on numerical simulations.

Before stating in Lemma~\ref{Lemma upper bound} a version of the upper bound on the ground-state energy that is useful for the setting discussed in this section, and subsequently our convergence result in Theorem~\ref{Theorem convergence}, let us recall the definitions of $\widetildesinels(j)$ and $\widetildesiners(j)$ from \eqref{Definition cossinus tilde functionen}, where we now have $r_{\max} = r_{\min} = 0$. Thus in particular,  
\begin{equation}\label{equation phi infty Hauptteil}
\widetilde \cos_{k,0}^{\text{l}}(j) + \widetilde \cos_{k,0}^{\text{r}}(j)
= \left(2/|V_k|\right)^{1/2} \cos \left( (k-|j|+ \tfrac{1}{2}) \pi /|V_k| \right) =  \varphi_{k,\infty}(j)
\end{equation}
for $j \in V_k$ and all $k \in \mathds N$, see also \eqref{equation phi infty} in the appendix. In addition, we write
\begin{equation*}
a_k \coloneq a_{k,1} + a_{k,2} \ ,
\end{equation*}
with $a_{k,1}$ and $a_{k,2}$ from \eqref{definition ak1} and \eqref{definition ak2}, respectively.

\begin{lemma}[Upper bound ground state energy II] \label{Lemma upper bound}
Consider a Hamiltonian $H_{k,\alpha}$, $k \in \mathds N$, with arbitrary $\alpha > 0$. Then, for all but finitely many $k \in \mathds{N}$,
\begin{align} \label{inequality upper bound ground state energyI}
    \lambda_0(V_k, \alpha \delta_0) \le (1-\beta_k)   \lambda_0(V_k,\infty \delta_0) + \alpha |\varphi_{k, \alpha}(0)|^2 \ ,
\end{align}
where $\beta_k$ is such that $0 \le \beta_k \le 1$ and 
\begin{equation} \label{trail state Lemma 10}
 \psi_k(j):= \left( 1- \beta_k \right)^{1/2} \varphi_{k,\infty}(j) + \varphi_{k, \alpha}(0) \ , \quad j \in V_k \ , \quad k \in \mathds N \ ,
\end{equation}
is normalized. Furthermore, 
\begin{equation} \label{Inequality betak ak}
\beta_k \le a_{k} \leq  (const.) \alpha^{-1} k^{-1}
\end{equation}
for all but finitely many $k \in \mathds{N}$. 
\end{lemma}
\begin{proof}
Relation~\eqref{inequality upper bound ground state energyI} follows immediately by using the trial state \eqref{trail state Lemma 10} in combination with the minmax-principle. Calculating the norm of $\psi_k$ gives
 \begin{align}\label{Definitionbk0}
\beta_k & = 2 \left( 1 - \beta_k \right)^{1/2} \varphi_{k,\alpha}(0) \sum\limits_{j=-k}^k \varphi_{k,\infty}(j) + (2k+1) |\varphi_{k, \alpha}(0)|^2
\end{align}
for each $k \in \mathds N$. Since the last term on the right-hand side converges to zero as $k \rightarrow \infty$ by \eqref{upper bound phikalpha}, the existence of a solution follows by the intermediate-value theorem. 
The claimed upper bound on $\beta_k$ follows by comparing \eqref{inequality upper bound ground state energyI} with \eqref{LowerBoundEigenvalue}, and by subsequently taking~\eqref{UpperBoundAKs} into account.
\end{proof}
We now establish the main result of this section.
\begin{theorem}[Convergence result] \label{Theorem convergence}
    Consider the Hamiltonian $H_{k,\alpha}$, $k \in \mathds N$, with arbitrary $\alpha > 0$. Then
    \begin{equation*}
        \lim_{k \rightarrow \infty} |V_k|^3 \, \Gamma(V_k, \alpha) =\frac{8\pi^2}{\alpha}\ .
    \end{equation*}
\end{theorem}
\begin{proof}
Firstly, with Theorem~\ref{Theorem lower bound} and Lemma~\ref{Lemma upper bound}, while also taking into account \eqref{lambda1alphadelta0 equation}, we conclude
    \begin{align*}
    |V_k|^3 \, \Gamma(V_k, \alpha) + \alpha |V_k|^3 |\varphi_{k, \alpha}(0)|^2 - |V_k|^3 a_k \lambda_0(V_k,\infty \delta_0) & \le 0 \\
    \shortintertext{and}
    |V_k|^3 \, \Gamma(V_k, \alpha) + \alpha |V_k|^3 |\varphi_{k, \alpha}(0)|^2 - |V_k|^3 \beta_k \lambda_0(V_k,\infty \delta_0) & \ge 0 \ .
    \end{align*}
  We set
    \begin{align} \label{definition ck}
        c_k \coloneq |V_k|^3 \, \Gamma(V_k, \alpha) + \alpha |V_k|^3 |\varphi_{k, \alpha}(0)|^2
    \end{align}
    for all but finitely many $k \in \mathds N$. Due to Corollary~\ref{C_1 proof conjecture special case J = 0} we already know that $c_k > 0$ for all $k \in \mathds N$. Furthermore, we conclude
    \begin{align} \label{inequality for ck}
        |V_k|^3 \lambda_0(V_k, \infty \delta_0) \beta_k \le c_k \le |V_k|^3 \lambda_0(V_k, \infty \delta_0) a_k
    \end{align}
    for all but finitely many $k \in \mathds N$. 

Next, using \eqref{obere Schranke lambda0inftydelta0} we show that $\lim_{k \to \infty} k(a_k - \beta_k)=0$: Let $k \in \mathds N$ be sufficiently large but otherwise arbitrary. Recall that $a_k \ge \beta_k$, see~\eqref{Inequality betak ak}. For $J = \{ 0 \}$, 
\eqref{equality for ak} reads 
\begin{align*}
a_k & = |\varphi_{k, \alpha}(0)|^2  + 2 \varphi_{k, \alpha}(0) \sum\limits_{j=-k}^{-1} \varphi_{k, \alpha}(j) + 2 \varphi_{k, \alpha}(0) \sum\limits_{j=1}^k \varphi_{k, \alpha}(j) - 2 |\varphi_{k, \alpha}(0)|^2 k \\
& \le 2 \varphi_{k, \alpha}(0) \sum\limits_{j=-k}^{k} \varphi_{k, \alpha}(j) \ ,
\end{align*}
and with \eqref{Definitionbk0} we have
\begin{align*}
\beta_k & \ge 2 \left( 1 - \beta_k \right) \varphi_{k,\alpha}(0) \sum\limits_{j=-k}^k \varphi_{k,\infty}(j)  \ .
\end{align*}
Thus, we conclude
\begin{equation*}
0 \le k(a_k - \beta_k) \leq 2\varphi_{k, \alpha}(0)k\left(\sum_{j=-k}^{k}\varphi_{k, \alpha}(j)-\sum_{j=-k}^{k} \varphi_{k,\infty}(j) \right) + 2 k \beta_k  \varphi_{k,\alpha}(0) \sum\limits_{j=-k}^k \varphi_{k,\infty}(j) \ .
\end{equation*}
Using \eqref{Inequality betak ak}, Hölder's inequality, and \eqref{upper bound phikalpha}, one concludes that the second term on the right-hand side of the above inequality converges to zero.
Regarding the first term on the right-hand side of the above inequality, the eigenvalue equation and a straightforward calculation, also using that $\varphi_{k,\alpha}$ and $\varphi_{k,\infty}$ are symmetric about zero, imply
\begin{equation*}\begin{split}
\lambda_0(V_k, \alpha \delta_0)\sum_{j=-k}^{k}\varphi_{k, \alpha}(j)=\sum_{j=-k}^k \left( \mathcal L_k \varphi_{k,\alpha} \right) (j)  + \alpha \varphi_{k,\alpha}(0)=\alpha \varphi_{k, \alpha}(0)
\end{split}
\end{equation*}
as well as 
\begin{equation} \label{equation for varphi infty one}
\lambda_0(V_k,\infty \delta_0)\sum_{j=-k}^{k}\varphi_{k,\infty}(j)=2 \sum_{j=1}^k \left( \mathcal L_k \varphi_{k,\infty} \right) (j) = 2 (\varphi_{k,\infty}(1)- \varphi_{k,\infty}(0)) =2\varphi_{k,\infty}(1)\ .
\end{equation}
%
Hence, we conclude 
\begin{equation*}\begin{split}
& 2\varphi_{k, \alpha}(0)k\left(\sum_{j=-k}^{k}\varphi_{k, \alpha}(j)-\sum_{j=-k}^{k} \varphi_{k,\infty}(j) \right) \\
\le & 2\varphi_{k, \alpha}(0)k \lambda_0^{-1}(V_k,\alpha \delta_0) \lambda^{-1}_0(V_k,\infty \delta_0)  \\
&  \cdot \Big[\lambda_0(V_k,\infty \delta_0)\big(\alpha\varphi_{k, \alpha}(0)-2\varphi_{k,\infty}(1) \big) + 2 \varphi_{k,\infty}(1) \big(\lambda_0(V_k,\infty \delta_0)- \lambda_0(V_k, \alpha \delta_0) \big) \Big]\ . \\
\end{split}
\end{equation*}
Due to Proposition~\ref{Result with phi 1 phi alpha0} of the appendix, the first term in the brackets is bounded from above by zero. Regarding the remaining term, we recall that $\lambda_0(V_k,\infty \delta_0)- \lambda_0(V_k, \alpha \delta_0)=\Gamma(V_k, \alpha)$, and use Theorem~\ref{Theorem bounds accumulation points general potential}, that $\lambda_0(V_k, \alpha \delta_0), \lambda_0(V_k,\infty \delta_0) \sim k^{-2}$, and \eqref{upper bound phikalpha}. Also, note that $\varphi_{k,\infty}(1) \le (const.) k^{-3/2}$ by \eqref{equation for varphi infty one} and Hölder's inequality. Overall, we obtain 
$0 \le k(a_k - \beta_k) \lesssim k^{-1}$. 

This now implies that
\begin{align*}
    \lim\limits_{k \to \infty} \left( c_k - |V_k|^3 \lambda_0(V_k,\infty \delta_0) \beta_k \right) = 0 \ .
\end{align*}
As a next step, we use \eqref{Definitionbk0} to first obtain
\begin{align*}
    & |V_k|^3 \lambda_0(V_k, \infty \delta_0) \beta_k \\
    = & \left(|V_k|^2 \lambda_0(V_k,\infty \delta_0)\right) \left( |V_k| 2 \left( 1 - \beta_k \right)^{1/2} \varphi_{k,\alpha}(0) \sum\limits_{j=-k}^k \varphi_{k,\infty}(j) + |V_k| (2k+1) |\varphi_{k, \alpha}(0)|^2\right) \ .
\end{align*}
The last term in the large parenthesis converges to zero, see  \eqref{upper bound phikalpha}. For the remaining term, we have, with \eqref{obere Schranke lambda0inftydelta0}, \eqref{equation phi infty Hauptteil}, \eqref{Inequality betak ak}, \eqref{equation for varphi infty one}, since $\cos(x) = - \sin(x - \pi/2)$ for all $x \in \mathds R$, and with Taylor expansion of $\sin(x)$ about zero, as well as \eqref{eq:eigfunctionformula}, Lemma~\ref{lem:asykappa}, and Propositions~\ref{Proposition norm vk abschaetzung nach unten} and \ref{Proposition norm vk abschaetzung nach oben} from the appendix,
\begin{align*}
    & 2 \left( 1 - \beta_k \right)^{1/2} |V_k|^3  \varphi_{k,\alpha}(0) \lambda_0(V_k, \infty \delta_0) \sum_{j=-k}^k \varphi_{k,\infty}(j) \\
     = & \, 4 \sqrt{2} \left( 1 - \beta_k \right)^{1/2} |V_k|^{5/2} \left(k^{-1/2}(1+o(1))\right) \cos\left((k + \tfrac{1}{2}) \pi/\kappa_0\right) \cos( (k-1 + \tfrac{1}{2}) \pi / |V_k|)\\
    = & \left( \dfrac{2|V_k|}{k} \right)^{1/2}  \left( 1 - \beta_k \right)^{1/2} |V_k|^3 \Gamma(V_k,\alpha) + o(1)\ .
\end{align*}
Similarly we obtain
    %
\begin{equation*}
    |V_k|^3 |\varphi_{k, \alpha}(0)|^2 = \dfrac{1}{16 \pi^2} \gamma_k\left( |V_k|^3 \, \Gamma(V_k, \alpha) \right)^2 + o(1) \ ,
\end{equation*}
    %
    %
    %
    %
    where $\gamma_k \coloneq |V_k| \|v_{k, \alpha}\|_{\mathcal H_k}^{-2} \left( 1 + o(1) \right)^{-2}$, $k \in \mathds N$, converges to 2, by Propositions~\ref{Proposition norm vk abschaetzung nach unten} and~\ref{Proposition norm vk abschaetzung nach oben}.
    Thus,
    \begin{align*}
        0 & = \lim\limits_{k \to \infty} \left( c_k - |V_k|^3 \lambda_0(V_k,\infty \delta_0) \beta_k \right) \\
        & = \lim\limits_{k \to \infty} \left( \dfrac{\alpha}{16 \pi^2} \gamma_k\left( |V_k|^3 \, \Gamma(V_k, \alpha) \right)^2 -  \left(  \left(  \dfrac{2|V_k|}{k} \right)^{1/2} \left( 1 - \beta_k \right)^{1/2} - 1 \right) |V_k|^3 \Gamma(V_k,\alpha) + o(1) \right) \ .
    \end{align*}

    Lastly,
    we now conclude, also using \eqref{Inequality betak ak},
    \begin{align*}
        \dfrac{\alpha}{8\pi^2} \left( \liminf\limits_{k \to \infty} |V_{k}|^3 \Gamma(V_{k},\alpha) \right)^2 - \liminf\limits_{k \to \infty} |V_{k}|^3 \Gamma(V_{k},\alpha) = 0
    \end{align*}
and therefore, since $\liminf_{k \to \infty} |V_{k}|^3 \Gamma(V_{k},\alpha) > 0$ by Corollary~\ref{C_1 proof conjecture special case J = 0},
\begin{align*}
        \liminf\limits_{k \to \infty} |V_{k}|^3 \Gamma(V_{k},\alpha) = \dfrac{8\pi^2}{\alpha} \ .
    \end{align*}
Similarly, 
%
%
%
%
$\limsup_{k \to \infty} |V_{k}|^3 \Gamma(V_{k},\alpha) = 8\pi^2/\alpha$,
and thus in total
\begin{align*}
    \lim\limits_{k \to \infty} |V_{k}|^3 \Gamma(V_{k},\alpha) = \dfrac{8 \pi^2}{\alpha} \ .
\end{align*}
\end{proof}
Finally, it is interesting to mention that the limit obtained in Theorem~\ref{Theorem convergence} is in agreement with the numerical results presented in \cite{KY}, which were obtained for the cases $\alpha=1$ and $\alpha=10$.

	
	\vspace*{0.5cm}
	
	{\small
		\bibliographystyle{amsalpha}
		\bibliography{Literature}}

\appendix 

\section{Auxiliary results} \label{Appendix}

\noindent In this appendix, we provide auxiliary results that we used to prove Theorem~\ref{Theorem convergence}; the main result of the appendix is Proposition~\ref{Result with phi 1 phi alpha0}. We recall that Theorem~\ref{Theorem convergence} is concerned with the special case where $J=\{0\}$, meaning the external potential is localized at the zero vertex. Therefore, the results of the appendix are also concerned with this special case. Also recall that $\lambda_0(V_k, \alpha \delta_0) > 0$ denotes the lowest eigenvalue of $H_{k,\alpha} = \mathcal L_k + \alpha \delta_0$, $\alpha >0$, and that $|V_k| = 2k+1$. For convenience, we also write $\lambda_0= \lambda_0(V_k, \alpha \delta_0)$ in this appendix and assume $k \in \mathds N$ to be sufficiently large so that $\lambda_0 < 2$.

Let $\kappa_0= \kappa_0(V_k, \alpha) > 2$ be such that 
\begin{equation}\label{eq:relationcos}
    2- 2\cos\left(\pi/\kappa_0\right) = \lambda_0 \ .
\end{equation}
Note that one may interpret $\kappa_0$ as an effective half-period of the cosine, where the $\alpha= \infty$ case corresponds to $\kappa_0(V_k,\infty)= |V_k|$ and the $\alpha = 0$ case corresponds, informally, to $\kappa_0(V_k,0)= \infty$. This behavior is consistent with the properties of the eigenfunctions in these limiting cases. In the next lemma, we discuss the relationship between $\kappa_0$ and $|V_k|$, and also relate $\kappa_0$ with the spectral gap $\Gamma(V_k, \alpha)$.

\begin{lemma}\label{lem:asykappa}
    For all but finitely many $k \in \mathds N$, we have
    \begin{equation}\label{EquationKappa}
        \kappa_0(V_k, \alpha)= |V_k| + \frac{1}{2\pi^2} |V_k|^{3} \Gamma(V_k, \alpha) + \mathcal O\left( |V_k|^{-1} \right) \ .
    \end{equation}
\end{lemma}
\begin{proof}Th
    Starting with \eqref{eq:relationcos} and using the equation $\cos(2x) = 1 - 2 \sin^2(x)$ we conclude 
    \begin{equation*}
         \kappa_0(V_k,\alpha) = \dfrac{\pi/2}{\arcsin\left( \sqrt{\lambda_0}/2 \right)}
    \end{equation*}
   for all but finitely many $k \in \mathds N$.
   In a next step we expand the $\arcsin$ about zero, as well as use  that $\lambda_0(V_k, \alpha \delta_0) = \lambda_1(V_k,\alpha \delta_0)  - \Gamma(V_k,\alpha)$ and $\lambda_1(V_k,\alpha \delta_0)  = \pi^2 |V_k|^{-2} + \mathcal O(|V_k|^{-4})$, see~\eqref{lambda1alphadelta0 equation} and~\eqref{obere Schranke lambda0inftydelta0}, to obtain
    \begin{align*}
        \kappa_0(V_k,\alpha) 
        &= \frac{|V_k|}{\sqrt{1- \pi^{-2} |V_k|^{2} \Gamma(V_k, \alpha) + \mathcal O(|V_k|^{-2})} + \mathcal O \left( |V_k|^{-2} \right)} \ .
    \end{align*}
    Finally, expanding $\sqrt{1-x}$ and of $(1-x)^{-1}$ about $x=0$, we conclude 
    \begin{align*}
        \kappa_0(V_k,\alpha) = |V_k| + (2\pi^2)^{-1} |V_k|^3 \Gamma(V_k, \alpha) + \mathcal O\left( |V_k|^{-1} \right) \ ,
    \end{align*}
    for all but finitely many $k \in \mathds N$.
\end{proof}
An immediate consequence of Corollary~\ref{C_1 proof conjecture special case J = 0} and Lemma~\ref{lem:asykappa} is now the following result.
\begin{cor}\label{cor:asykappa}
    There exist $C_1 = C_1(\alpha), C_2=C_2(\alpha) >0$ such that for all but finitely many $k \in \mathds N$,
    \begin{align*}
         C_1 \le \kappa_0(V_k, \alpha) - |V_k| \le   C_2 \ .
    \end{align*}
\end{cor}
Next, we introduce the auxiliary function
\begin{gather*}
    u_{k,\alpha}(j) \coloneq \cos\left((j+\tfrac{1}{2}) \pi/\kappa_0\right) \ , \quad j \in \mathds Z \ .
\end{gather*}
For all $j\in \mathds Z$, using the equation $\cos(x+y) = \cos(x)\cos(y) - \sin(x)\sin(y)$, we then have
\begin{align*}
    u_{k,\alpha}(j+1) &= \cos\left( (j+\tfrac{1}{2}) \pi/\kappa_0\right)\cos \left( \pi/\kappa_0 \right) - \sin\left((j + \tfrac{1}{2}) \pi/\kappa_0 \right)  \sin\left( \pi/\kappa_0\right)\ ,\\
\shortintertext{and}
    u_{k,\alpha}(j-1) &= \cos\left( ( j + \tfrac{1}{2})  \pi/\kappa_0\right) \cos\left( \pi/\kappa_0\right) + \sin\left((j + \tfrac{1}{2}) \pi/\kappa_0 \right)  \sin\left( \pi/\kappa_0\right) \ ,
\end{align*}
and therefore
\begin{equation}\label{eigenfunction constraction equation part 1}
    - u_{k,\alpha}(j-1) + 2 u_{k,\alpha}(j) - u_{k,\alpha}(j+1) = \left(2 - 2 \cos \left( \pi/\kappa_0 \right)\right) u_{k,\alpha}(j) = \lambda_0  u_{k,\alpha}(j) \ .
\end{equation}
Since $u_{k,\alpha}(-1)=u_{k,\alpha}(0)$ (by direct computation) and by using the above equation for $j=0$, we conclude
\begin{equation}\label{eigenfunction constraction equation part 2}
    u_{k,\alpha}(0) - u_{k,\alpha}(1) = -u_{k,\alpha}(-1) + 2u_{k,\alpha}(0) - u_{k,\alpha}(1) = \lambda_0 u_{k,\alpha}(0) \ .
\end{equation}
Now, choose $A(j)$ such that $ \varphi_{k, \alpha}(j) = A(j) v_{k,\alpha}(j)$ where $\varphi_{k, \alpha}$ is the eigenfunction to $\lambda_0$ and $v_{k,\alpha}(j):=u_{k,\alpha}(k-|j|)$. Then, using the eigenvalue equation for $\varphi_{k, \alpha}$ as well as~\eqref{eigenfunction constraction equation part 1} and~\eqref{eigenfunction constraction equation part 2}, one obtains that $A(j) = A(j-1)$ for $j\in \{1,\ldots, k\}$ and hence, also using that $v_{k,\alpha}$ is symmetric about zero, $A(j)=\|v_{k,\alpha}\|_{\mathcal H_k}^{-1}$, $j \in V_k$. Furthermore, from this one concludes that the (unique) eigenfunction of $H_{k,\alpha}$ to $\lambda_0$ is given by
\begin{equation}\label{eq:eigfunctionformula}
    \varphi_{k, \alpha}(j) = \|v_{k,\alpha}\|_{\mathcal H_k}^{-1} \, v_{k,\alpha}(j) \qquad \text{where} \qquad  v_{k,\alpha}(j) = \cos\left((k-|j|+\tfrac{1}{2}) \pi/\kappa_0\right) \ .
\end{equation}
\begin{proposition}\label{prop:diffalpha}
We have for all but finitely many $k \in \mathds N$,
\begin{gather*}
v_{k,\alpha}(1) - v_{k,\alpha}(0) = \frac{\pi}{|V_k|} - \frac{1}{2\pi} |V_k| \Gamma(V_k, \alpha) + \mathcal O(|V_k|^{-3}).
\end{gather*}
\end{proposition}
\begin{proof}
    Using Lemma~\ref{lem:asykappa} and Corollary~\ref{cor:asykappa}, which implies that $|\kappa_0 - 2k|$ is bounded from above by a constant, we obtain
    \begin{align*}
        v_{k,\alpha}(1) - v_{k,\alpha}(0) &= \cos\left( (k-\tfrac{1}{2}) \pi/\kappa_0 \right) - \cos\left( (k+\tfrac{1}{2}) \pi/\kappa_0 \right) = 2 \sin\left( \frac{\pi k}{\kappa_0}\right) \sin\left( \frac{\pi}{2\kappa_0}  \right)\\
        &= 2 \cos \left( \frac{\kappa_0 - 2k}{\kappa_0} \frac{\pi}{2} \right) \sin \left( \frac{\pi}{2\kappa_0} \right) = 2 \Big( 1 +\mathcal O\left(|V_k|^{-2}\right) \Big) \left(  \dfrac{\pi}{2\kappa_0} + \mathcal O\left(|V_k|^{-3}\right)\right) \\
        & = \dfrac{\pi}{\kappa_0}  +\mathcal O\left(|V_k|^{-3} \right) \ .
    \end{align*}
    Since $|V_k|^2\Gamma(V_k,\alpha) \sim k^{-1}$ by Corollary~\ref{C_1 proof conjecture special case J = 0}, we thus have 
    \begin{align*}
        v_{k,\alpha}(1) - v_{k,\alpha}(0) &= \frac{\pi}{|V_k|} \left( 1 - \frac{1}{2\pi^2} |V_k|^2 \Gamma(V_k, \alpha) +\mathcal O\left (|V_k|^{-2} \right) \right) + \mathcal O \left( |V_k|^{-3} \right)\ ,
    \end{align*}
    for all but finitely many $k \in \mathds N$.
\end{proof}

An immediate consequence of Proposition~\ref{prop:diffalpha} is the following inequality.
\begin{Corollary}\label{cor:diffalpha}
There exists a $C_3 = C_3(\alpha)>0$ such that for all but finitely many $k \in \mathds N$,
\begin{gather*}
    \frac{\pi}{|V_k|}- \left( v_{k,\alpha}(1) - v_{k,\alpha}(0) \right) \ge   \frac{C_3}{|V_k|^2} \ .
\end{gather*}
\end{Corollary}
We also have the following statement. 
\begin{proposition} \label{Proposition norm vk abschaetzung nach unten}  For all but finitely many $k \in \mathds N$, we have
    \begin{equation}
        \|v_{k,\alpha}\|_{\mathcal H_k}^2 \ge \dfrac{|V_k|}{2} \ .
    \end{equation}
\end{proposition}
\begin{proof}
        By Corollary~\ref{cor:asykappa} we have $\kappa_0 \ge |V_k|$ for all but finitely many $k \in \mathds N$.
        Thus, using $\cos^2(x) = (1/2)(1 + \cos(2x))$, we conclude
        \begin{align*}
            \| v_{k,\alpha} \|_{\mathcal H_k}^2 &= \sum_{j=-k}^k \cos^2 \left ( (k- |j| + \tfrac{1}{2}) \pi /\kappa_0 \right ) \ge \sum_{j=-k}^k \cos^2((k-|j|+ \tfrac{1}{2}) \pi /|V_k|)\\
            &= \dfrac{|V_k|}{2} + \dfrac{1}{2} \sum_{j=-k}^k \cos\left( \dfrac{|V_k|-2|j|}{|V_k|} \pi \right) = \dfrac{|V_k|}{2} - \dfrac{1}{2} \sum_{j=-k}^k \cos\left ( \frac{2\pi j}{|V_k|}\right )\\
            &= \dfrac{|V_k|}{2} - \dfrac{1}{2}  \mathrm{Re} \sum_{j=-k}^k \exp\left( 2\pi \mathrm i \frac{j}{|V_k|}\right) = \frac{|V_k|}{2}
        \end{align*}
    for all but finitely many $k \in \mathds N$.
\end{proof}
Proposition~\ref{Proposition norm vk abschaetzung nach unten} is complemented by the following result.
\begin{proposition} \label{Proposition norm vk abschaetzung nach oben}
    For all but finitely many $k\in \mathbb N$ we have
    \begin{equation*}
        \|v_{k,\alpha}\|_{\mathcal H_k}^2 \le \dfrac{|V_k|}{2} + \lceil C_2/2 \rceil \ ,
    \end{equation*}
    with $C_2 > 0$ from Corollary~\ref{cor:asykappa}.
\end{proposition}
\begin{proof}
    By Corollary~\ref{cor:asykappa} we have $\kappa_0 \le |V_k| + C_2$, and similarly as in Proposition~\ref{Proposition norm vk abschaetzung nach unten} we have, for all but finitely many $k\in \mathds N$,
    \begin{align*}
            \| v_{k,\alpha} \|_{\mathcal H_k}^2 &= \sum_{j=-k}^k \cos^2 \Big( (k- |j| + \tfrac{1}{2}) \pi /\kappa_0 \Big) \\
            &\le \sum_{j=-(k+ \lceil C_2/2\rceil)}^{k+\lceil C_2/2\rceil } \cos^2 \Big( \big( k+ \lceil C_2/2\rceil- |j|+ \tfrac{1}{2} \big) \pi/ \big( 2(k+ \lceil C_2/2\rceil)+1 \big) \Big)\\
            & = k+ \lceil C_2/2 \rceil + 1/2 \ .
        \end{align*}
\end{proof}

Next, we define the function $v_{k,\infty}$ via
\begin{gather*}
    v_{k,\infty}(j) = \cos( (k-|j| + \tfrac{1}{2}) \pi / |V_k|) \ ,
\end{gather*}
and conclude $\varphi_{k,\infty}(j) = v_{k,\infty}(j)/ \|v_{k,\infty}\|_{\mathcal H_k}$ (recall that $\varphi_{k,\infty}$ is the ground-state eigenfunction to $H_{k,\alpha}$ for $\alpha=\infty$). As computed in the proof of Proposition~\ref{Proposition norm vk abschaetzung nach oben},
\begin{gather*}
    \|v_{k,\infty}\|_{\mathcal H_k}^2 = \sum_{j=-k}^k \cos^2 \left( (k-|j|+ \tfrac{1}{2}) \pi/|V_k| \right) = \frac{|V_k|}{2} \ .
\end{gather*}
Hence, we have
\begin{gather} \label{equation phi infty}
     \varphi_{k,\infty}(j) = \left(2/|V_k|\right)^{1/2} \cos \left( (k-|j|+ \tfrac{1}{2}) \pi /|V_k| \right) \ .
\end{gather}
We now establish the main result of the appendix.
\begin{proposition} \label{Result with phi 1 phi alpha0}
    There exists a constant $C_4 = C_4(\alpha)>0$ such that for all but finitely many $k \in \mathds N$, we have
    \begin{gather*}
       2  \varphi_{k,\infty}(1)-\alpha \varphi_{k, \alpha}(0) \ge C_4 |V_k|^{-5/2} \ .
    \end{gather*}
\end{proposition}
\begin{proof}
    Using Corollary~\ref{cor:diffalpha} and Proposition~\ref{Proposition norm vk abschaetzung nach unten} we obtain, for all but finitely many $k$,
    \begin{align*}
        \varphi_{k, \alpha}(1) - \varphi_{k, \alpha}(0) &= \frac{v_{k,\alpha}(1) - v_{k,\alpha}(0)}{\|v_{k,\alpha}\|_{\mathcal H_k}} \le \frac{2^{1/2}}{|V_k|^{1/2}} \left( \frac{\pi}{|V_k|}  -\frac{C_3}{|V_k|^2} \right) = \frac{2^{1/2}\pi}{|V_k|^{3/2}}  \left(1- \frac{C_3}{\pi|V_k|} \right) \ .
    \end{align*}
Using \eqref{equation phi infty} then yields
    \begin{align*}
        \varphi_{k, \alpha}(1) - \varphi_{k, \alpha}(0) &\leq \varphi_{k,\infty}(1)- 2^{1/2} C_3 |V_k|^{-5/2} + \mathcal O \left( |V_k|^{-7/2} \right) \ .
    \end{align*}
    Since $\lambda_0 \, \varphi_{k, \alpha}(0) \lesssim |V_k|^{-7/2}$, see also \eqref{upper bound phikalpha}, there exists a constant $C_4>0$ such that for all but finitely many $k \in \mathds N$,
    \begin{gather*}
         \varphi_{k,\infty}(1) - \left( \varphi_{k, \alpha}(1) - \varphi_{k, \alpha}(0) \right) - \frac{\lambda_0}{2} \varphi_{k, \alpha}(0) \ge \frac{C_4}{2}|V_k|^{-5/2} \ .
    \end{gather*}
    Finally, employing the eigenvalue equation for $\varphi_{k, \alpha}$ and the fact that $\varphi_{k, \alpha}$ is symmetric about zero yields
    \begin{gather*}
        2 \varphi_{k,\infty}(1) -\alpha \varphi_{k, \alpha}(0) = 2 \varphi_{k,\infty}(1) - 2( \varphi_{k, \alpha}(1) - \varphi_{k, \alpha}(0)) - \lambda_0 \, \varphi_{k, \alpha}(0) \ge C_4 |V_k|^{-5/2}
    \end{gather*}
    for all but finitely many $k \in \mathds N$.
\end{proof}

\end{document}